\documentclass{article}
\usepackage{amsmath, amssymb, amsthm, hyperref, graphicx, enumerate}
\usepackage[margin=1.25in]{geometry}

\newtheorem{thm}{Theorem}[section]
\newtheorem{prop}{Proposition}[section]

\newtheorem{cor}{Corollary}[section]

\theoremstyle{definition}

\newtheorem{rmk}{Remark}[section]
\newtheorem{question}{Question}[section]

\title{Convergence of Siegel-Veech constants} 
\author{Benjamin Dozier \thanks{Department of Mathematics, Stanford University, \href{mailto:bdozier@stanford.edu}{\nolinkurl{bdozier@stanford.edu}}.  Supported in part by  NSF grant DGE-114747.}}
\date{}
\begin{document}
\maketitle

\begin{abstract}
  We show that for any weakly convergent sequence of ergodic $SL_2(\mathbb{R})$-invariant probability measures on a stratum of unit-area translation surfaces, the corresponding Siegel-Veech constants converge to the Siegel-Veech constant of the limit measure.  Together with a measure equidistribution result due to Eskin-Mirzakhani-Mohammadi, this yields the (previously conjectured) convergence of sequences of Siegel-Veech constants associated to Teichm\"uller curves in genus two.  

The proof uses a recurrence result closely related to techniques developed by Eskin-Masur.  We also use this recurrence result to get an asymptotic quadratic upper bound, with a uniform constant depending only on the stratum, for the number of saddle connections of length at most $R$ on a unit-area translation surface.  
 
\end{abstract}

\section{Introduction}
\label{sec:intro}
\subsection{Setting}

\paragraph{Basic Definitions.} A \emph{translation surface} is a pair $X=(M,\omega)$, where $M$ is a Riemann surface, and $\omega$ is a holomorphic $1$-form.  Away from its zeroes, $\omega$ defines a flat (Euclidean) metric. The metric has a conical singularity of cone angle $2(n+1)\pi$ at each zero of order $n$.  

A \emph{saddle connection} is a geodesic segment that starts and ends at zeroes (we allow the endpoints to coincide), with no zeroes on the interior of the segment.  We can also consider closed loops not hitting zeroes that are geodesic with respect to the flat metric.  Whenever there is one of these, there will always be a continuous family of parallel closed geodesic loops with the same length.  We refer to a maximal such family as a \emph{cylinder}.  Every cylinder is bounded by a union of saddle connections parallel to the cylinder.

The bundle $\Omega \mathcal{M}_g$ of holomorphic 1-forms over $\mathcal{M}_g$ (the moduli space of genus $g$ Riemann surfaces), with zero section removed, can be thought of as the moduli space of translation surfaces.  This bundle breaks up into strata of translation surfaces that have the same multiplicities of the zeroes of $\omega$.  We denote by $\mathcal{H}(m_1,\dots,m_k)$ the stratum of unit-area surfaces with $k$ zeroes of order $m_1,\ldots,m_k$. 

There is an action of $SL_2(\mathbb{R})$ on each stratum $\mathcal{H}$ which will play a central role in our discussion.  To see the action, we first observe that by cutting along saddle connections, we can represent every translation surface as a set of polygons in the plane, such that every side is paired up with a parallel side of equal length.  Since $SL_2(\mathbb{R})$ acts on polygons in the plane, preserving the property of a pair of sides being parallel and equal length, the group acts on $\mathcal{H}$.  We will work mostly with elements of the following form:
\begin{eqnarray*}
  g_t =
  \left(\begin{matrix}
    e^{-t} & 0 \\
    0 & e^t
  \end{matrix}\right), \text{  \ \    } r_{\theta} = \left(
        \begin{matrix}
          \cos \theta & -\sin\theta \\
          \sin \theta & \cos\theta
        \end{matrix}\right).
\end{eqnarray*}

\paragraph{Invariant measures.} On each $\mathcal{H}$ there is a canonical probability measure $\mu_{MV}$ in the Lebesgue measure class, called the \emph{Masur-Veech measure}, which is $SL_2(\mathbb{R})$-invariant and in fact ergodic (\cite{masur1982}, \cite{veech1982}).  This measure is defined in terms of the periods of the $1$-form $\omega$.  

There is a rich interplay between dynamics on an individual translation surface $X$ (e.g. properties of saddle connections or cylinders) and the dynamics of the $SL_2(\mathbb{R})$ action on strata.  In particular, by the seminal work of Eskin-Mirzakhani \cite{em2013} and Eskin-Mirzakhani-Mohammadi \cite{emm2015}, the orbit closure $\overline{SL_2(\mathbb{R})X}$ supports a canonical $SL_2(\mathbb{R})$-invariant ergodic probability measure, and properties of this measure are closely connected to dynamics on individual translation surfaces.

\paragraph{Siegel-Veech constants.}  Let $N(X,R)$ be the number of cylinders on $X$ of length at most $R$.  The study of asymptotics of this function as $R\to\infty$ is of central importance and has inspired much of the work on spaces of translation surfaces.  By work of Masur (\cite{masur1988} and \cite{masur1990}), for a fixed $X$, there are quadratic upper and lower bounds for the growth of $N(X,R)$ in terms of $R$.  If $\mu$ is an ergodic $SL_2(\mathbb{R})$-invariant probability measure on $\mathcal{H}$, then by Eskin-Masur (\cite{em2001}) there exists a constant $c(\mu)$, the \emph{cylinder Siegel-Veech constant} associated to $\mu$, characterized by the property that for $\mu$-a.e. $X$ in $\mathcal{H}$, 
$$N(X,R) \sim c(\mu) \cdot R^2$$
as $R\to \infty$.  It is an important open question whether for \emph{every} translation surface $X$, there exists some $c$ such that $N(X,R) \sim c R^2$, as $R\to\infty$.  This question is closely connected to the classification of measures on $\mathcal{H}$ that are invariant under the unipotent subgroup $\left\{ \left(\begin{matrix}
    1 & * \\
    0 & 1
  \end{matrix}\right) \right\} \subset SL_2(\mathbb{R})$. 

For more information on translation surfaces, the reader can consult one of many surveys available on the topic, for instance \cite{zorich2006} or \cite{wright2015}.

\subsection{Convergence of Siegel-Veech constants}

The first new result states that if measures converge, then the corresponding Siegel-Veech constants do as well.  

\begin{thm} Suppose $\mu_1,\mu_2,\ldots$ are ergodic $SL_2(\mathbb{R})$-invariant probability measures on $\mathcal{H}$, and that $\mu_n \to \eta$, in the weak-* topology, where $\eta$ is another ergodic $SL_2(\mathbb{R})$-invariant probability measure.  Then the Siegel-Veech constants satisfy $c(\mu_n) \to c(\eta)$. 
\label{thm:conv}
\end{thm}

In Section \ref{sec:applications}, we use Theorem \ref{thm:conv} to prove convergence, for the stratum $\mathcal{H}(2)$, of Siegel-Veech constants for non-arithmetic Teichm\"uller curves (numerical evidence of this was found by Bainbridge \cite{bainbridge2007}), and for arithmetic Teichm\"uller curves (conjectured by Leli{\`e}vre, based on numerical evidence and proof in a restricted case \cite{lelievre2006}).   This involves a new type of application of the measure equidistribution result of Eskin-Mirzakhani-Mohammadi in \cite{emm2015}. 

\begin{rmk} 
We will work with the cylinder Siegel-Veech constant for concreteness, but the result and proof work for other Siegel-Veech constants as well, in particular for the saddle connection Siegel-Veech constant (which counts saddle connections rather than cylinders) and the area Siegel-Veech constant (which is formed from counts of cylinders weighted by the area of the cylinder).   The result and proof also works with a pair $(M,q)$, where $q$ is a holomorphic \emph{quadratic} differential (also known as a \emph{half-translation surface}).  
\end{rmk}

\begin{rmk}
  Note that if we define the Siegel-Veech constant of $X$ to be the Siegel-Veech constant of the canonical measure whose support is the orbit closure $\overline{SL_2(\mathbb{R})X}$, then this does \emph{not} define a continuous function on $\mathcal{H}$.  This is because special surfaces with small orbit closure, for instance Veech surfaces, will often have Siegel-Veech constants different from $c(\mu_{MV})$, the Siegel-Veech constant for the Masur-Veech measure on $\mathcal{H}$ (see Section \ref{sec:previous-work} for references), while a dense subset of surfaces will have Siegel-Veech constant equal to $c(\mu_{MV})$. 
\end{rmk}

\subsection{Uniform asymptotic quadratic upper bound}

The second new result gives a \emph{uniform} quadratic upper bound on the number of cylinders, which holds, asymptotically, for \emph{all} surfaces in a stratum.  

\begin{thm}
  Given $\mathcal{H}$ a unit-area stratum, there exists a constant $c_{\max}$ such that for any surface $X\in \mathcal{H}$, 
$$N(X,R) \le c_{\max} R^2$$
for all $R\ge R_0(X)$, where $R_0: \mathcal{H} \to \mathbb{R}$ is an explicit function of the length of the shortest saddle connection on $X$ (and of the genus of the stratum).  
\label{thm:uniform}
\end{thm}

Note that the  function $R_0$ will not in general be bounded for a fixed stratum.  This is because a surface in a fixed stratum can have arbitrarily many short saddle connections (for instance, by taking a surface with a short slit, and then gluing in a cylinder with small height and circumference).  But, according to Theorem \ref{thm:uniform}, as we increase $R$, the effect of these short saddle connections eventually diminishes.  

\begin{rmk} 
As in the case of Theorem \ref{thm:conv}, the result and proof work if we replace the count of cylinders with the count of saddle connections, or the count of cylinders weighted by the area of the cylinder.  The result and proof also work with a pair $(M,q)$, where $q$ is a holomorphic quadratic differential.
\end{rmk}

\subsection{Recurrence result for the proofs}
The main tool needed in the proofs of both Theorem \ref{thm:conv} and Theorem \ref{thm:uniform} is the following proposition, which may be of independent interest.  It is a recurrence-type result which controls the length of the shortest saddle connection, on average, over translation surfaces on a large ``circle'' centered at any $X$.  Here we deduce the proposition directly from a more general result proved in \cite{dozier2017}; in that paper the more general result is used to study the distribution of angles of saddle connections.  

Let $\ell(X)$ denote the length of the shortest saddle connection on $X$, and let 
\begin{eqnarray*}
  g_t =
  \left(\begin{matrix}
    e^{-t} & 0 \\
    0 & e^t
  \end{matrix}\right), \text{  \ \    } r_{\theta} = \left(
        \begin{matrix}
          \cos \theta & -\sin\theta \\
          \sin \theta & \cos\theta
        \end{matrix}\right).
\end{eqnarray*}

\begin{prop}
 For any stratum $\mathcal{H}$ and $0<\delta<1/2$, there exists a function $\alpha: \mathcal{H} \to \mathbb{R}_{\ge 0}$ and constants $c_0,b$ such that for any $X\in \mathcal{H}$,
$$\int_0^{2\pi} \frac{1}{\ell(g_Tr_{\theta}X)^{1+\delta}}d\theta \le c_0 e^{-(1-2\delta) T} \alpha(X) + b,$$
for all $T>0$.  The function $\alpha(X)$ is bounded above by an explicit function of $\ell(X)$ (and the genus of the stratum).  
\label{prop:ineq}
\end{prop}

When we use Proposition \ref{prop:ineq}, it will be crucial that the constant $b$ does not depend on the surface $X$.  Some related results appear in \cite{em2001} (also see \cite{athreya2006}), but the above formulation with an additive constant which does not depend on the surface is new.  


\begin{proof}[Proof of Proposition \ref{prop:ineq}.]
  This is a special case of Proposition 2.1 in \cite{dozier2017}.  That Proposition involves integrating over any subinterval $I\subset [0,2\pi]$; the above is simply the case when $I$ equals $[0,2\pi]$.  The proof in that paper follows the strategy used by Eskin-Masur (\cite{em2001}), but keeps track of the constant $b$ above.  The approach is to use the ``system of integral inequalities'', which was first developed for the proof of the quantitative Oppenheim conjecture by Eskin-Margulis-Mozes (\cite{emm1998}), who were working in the context of lattices.  A key technical aspect unique to the translation surfaces context involves combining ``complexes'' of saddle connections.  
\end{proof}

\subsection{Previous work}
\label{sec:previous-work}
A good deal of progress has been made in understanding the Siegel-Veech constants of Veech surfaces, which often lead to explicit expressions for the quadratic growth rates for billiards on polygons.   In his foundational paper \cite{veech1989}, Veech used Eisenstein series to show that all Veech surfaces satisfy an exact quadratic asymptotic for the growth rate of cylinders, and he gave a way of computing the constants.  He computes the constants for translation surfaces arising from unfolding certain isosceles triangles.  Gutkin-Judge \cite{gj2000} give a different formula for computing the Siegel-Veech constant, the proof of which uses softer ergodic-theoretic results related to counting horocycles in the hyperbolic plane.  Vorobets \cite{vorobets1996} discovered similar results independently.  

Schmoll \cite{schmoll2002}  studied the problem of counting cylinders and saddle connections on tori with additional marked points.  Along similar lines, Eskin-Masur-Schmoll \cite{ems2003} study translation surfaces that are branched covers of tori.  Using Ratner theory, they get exact quadratic asymptotics for all these surfaces, and they explicitly compute the constants for certain surfaces arising from billiards in rectangles with barriers.  In complementary work, Eskin-Marklof-Morris \cite{emm2006} study the case of branched covers of Veech surfaces that are not tori.  They also get exact quadratic asymptotics for these surfaces, and they explicitly compute the constants for certain (non-Veech) surfaces that arise from unfolding triangles.  Their proof modifies the techniques of Ratner to work in their setting, where the relevant moduli space is not homogeneous, but shares some important properties with homogeneous spaces.  Bainbridge-Smillie-Weiss (\cite{bsw2016}) show that in the eigenforn loci in $\mathcal{H}(1,1)$, \emph{all} surfaces satisfy exact quadratic growth asymptotics.  

The results discussed above apply only to special translation surfaces.  In the opposite direction, one can ask about Siegel-Veech constants for the Masur-Veech measure on a whole stratum.  Eskin-Masur-Zorich \cite{emz2003} give a general method for computing these in terms of the volumes of strata and neighborhoods of certain parts of the boundary of strata.  Results of Eskin-Okounkov \cite{eo2001} allow one to compute these volumes.  

The convergence of the \emph{area} Siegel-Veech constants was proven by Matheus-M\"{o}ller-Yoccoz (\cite{mmy2015}, Remark 1.10).  That proof relies on the close connection between area Siegel-Veech constants and sums of Lyapunov exponents and does not work for other Siegel-Veech constants.

\subsection{Outline of the paper}
\begin{itemize}
\item Section \ref{sec:applications} gives an application to Siegel-Veech constants associated to Veech surfaces in genus $2$, an application showing that Siegel-Veech constants are bounded in a fixed stratum, and some results on the set of Siegel-Veech constants associated to all the measures on a stratum.  
\item Section \ref{sec:convergence} contains the proof of Theorem \ref{thm:conv} (Convergence of Siegel-Veech constants), assuming the recurrence-type result Proposition \ref{prop:ineq}.
\item Section \ref{sec:uniform} contains the proof of Theorem \ref{thm:uniform} (Uniform asymptotic quadratic upper bound), assuming Proposition \ref{prop:ineq}.
\item Section \ref{sec:bound} poses open questions about the size of the extremal Siegel-Veech constants in each stratum, and discusses some results in this direction.  
\end{itemize}

\subsection{Acknowledgements}
I would like to thank Maryam Mirzakhani, my thesis advisor, for guiding me with numerous stimulating conversations and suggestions.  I am also very grateful to Alex Wright, for many helpful discussions and detailed feedback.

\section{Applications}
\label{sec:applications}

\subsection{Convergence of Siegel-Veech constants in $\mathcal{H}(2)$}

We give an application of Theorem \ref{thm:conv} in genus two, which we then use to give proofs of the convergence of certain sequences of Siegel-Veech constants.  The application uses in a crucial way the equidistribution result of Eskin-Mirzakhani-Mohammadi \cite{emm2015}.

Let $\mu_{MV}$ be the Masur-Veech measure on the stratum $\mathcal{H}(2)$.

\begin{thm} 
Let $\{C_n\}$ be a sequence of distinct closed $SL_2(\mathbb{R})$-orbits in $\mathcal{H}(2)$, and let $\mu_n$ be the ergodic $SL_2(\mathbb{R})$-invariant probability measure whose support is $C_n$.  Then
 $$\lim_{n\to\infty} c(\mu_n) =c(\mu_{MV})=\frac{10}{\pi}.$$
\label{thm:h2}
\end{thm}

Recall that for surfaces generating a closed $SL_2(\mathbb{R})$ orbit (known as Veech surfaces), the generic quadratic growth constant for the whole orbit actually equals the constant for every surface in the orbit (\cite{veech1989}, Proposition 3.10).  Thus the above result implies that the quadratic growth constants for a sequence of distinct Veech surfaces in $\mathcal{H}(2)$ tend to the constant for the whole stratum.  

\begin{proof}
We claim that $\lim_{n\to\infty} \mu_n = \mu_{MV}$.  By the equidistribution result in \cite{emm2015} (Corollary 2.5), if this were not the case, there would exist a subsequence $k_n \to \infty$ and $\mathcal{N}$ an affine invariant submanifold of $\mathcal{H}(2)$, with each $C_{k_n}$ contained in $\mathcal{N}$.  (An affine invariant manifold is the image of a proper immersion from a connected manifold to a stratum that is cut out locally by homogeneous real linear equations in period coordinates).  Now $\mathcal{N}$ cannot be a single closed $SL_2(\mathbb{R})$-orbit, since it contains infinitely many distinct closed $SL_2(\mathbb{R})$-orbits.  Then by McMullen's classification (\cite{mcmullen2007}, Theorem 1.2), $\mathcal{N}$ must be the whole stratum $\mathcal{H}(2)$, contradiction, establishing the claim.

Now Theorem \ref{thm:conv} gives $\lim_{n\to\infty} c(\mu_n) = c(\mu_{MV})$.  By \cite{emz2003} (Example 14.7, second case), $c(\mu_{MV})=\frac{10}{\pi}$ (the normalization for the Siegel-Veech constant used in that paper differs from ours by a factor of $\pi$).  

\end{proof}

The two corollaries below, which apply to non-arithmetic and arithmetic Veech surfaces, respectively, follow immediately from Theorem \ref{thm:h2}.  

\begin{cor}[Convergence for non-arithmetic Veech surfaces]
  Let $D$ be a positive integer that is not a perfect square, with $D\equiv 0,1 \bmod 4$, and let $E_D$ be the $SL_2(\mathbb{R})$-orbit in $\mathcal{H}(2)$ of a pair $(M,\omega)$ for which the Jacobian $\operatorname{Jac}(M)$ admits real multiplication by $\mathcal{O}_D$, the ring of integers in $\mathbb{Q}[\sqrt{D}]$, with $\omega$ an eigenform (these orbits are known to be closed).  Let $\mu_D$ be the ergodic $SL_2(\mathbb{R})$-invariant probability measure whose support is $E_D$.  Then 
$$\lim_{D\to\infty} c(\mu_D) =c(\mu_{MV})=\frac{10}{\pi}.$$
\end{cor}

Bainbridge found a formula for $c(\mu_D)$ and numerical evidence suggesting that the above convergence holds (\cite{bainbridge2007}, discussion after Theorem 14.1).  

\begin{cor}[Convergence for arithmetic Veech surfaces]
  Let $\{S_n\}$ be a sequence of square-tiled surfaces in $\mathcal{H}(2)$, where $S_n$ is tiled by exactly $k_n$ squares, with $k_n\to\infty$.  Let $\mu_n$ be the ergodic $SL_2(\mathbb{R})$-invariant probability measure whose support is the (closed) orbit $SL_2(\mathbb{R})S_n$.  Then 
$$\lim_{n\to\infty} c(\mu_n) = c(\mu_{MV}) = \frac{10}{\pi}.$$
\end{cor}

This result was conjectured by Leli{\`e}vre, who proved it with the additional restriction that each $k_n$, the number of squares, is a prime, and found numerical evidence for the general case  (\cite{lelievre2006}).

\begin{rmk}
  One can also use the strategy above to prove that the Siegel-Veech constants corresponding to the eigenform loci in $\mathcal{H}(1,1)$ (which are no longer just closed orbits) converge to the Siegel-Veech constant for $\mathcal{H}(1,1)$.  For non-arithmetic eigenform loci, this was proven by Bainbridge; in fact the Siegel-Veech constants are the same for all the non-arithmetic eigenform loci, except the $D=5$ locus (\cite{bainbridge2010}, Theorem 1.5).  For the arithmetic eigenform loci, convergence was proven by Eskin-Masur-Schmoll (\cite{ems2003}, Theorem 1.3); here the sequence of Siegel-Veech constants is not eventually constant.   
\end{rmk}

\subsection{Boundedness of Siegel-Veech constants}
\label{sec:bound-sv}

\begin{thm} \label{thm:upperbound}
  Fix a stratum $\mathcal{H}$.  There exists a bound $B$ (depending on the stratum) such that for any ergodic $SL_2(\mathbb{R})$-invariant probability measure $\mu$ on $\mathcal{H}$,
$$c(\mu) \le B.$$
\end{thm}

We give two different proofs. 

\begin{proof}[Proof via Theorem \ref{thm:conv}]
  Suppose, for the sake of contradiction, that $\mu_n$ is a sequence of ergodic $SL_2(\mathbb{R})$-invariant probability measures on $\mathcal{H}$, with $c(\mu_n) \to \infty$.  By passing to a subsequence, and applying the equidistribution theorem \cite{emm2015} (Corollary 2.5), we can assume that $\mu_n \to \eta$, where $\eta$ is another ergodic $SL_2(\mathbb{R})$-invariant probability measure.  Then Theorem \ref{thm:conv} gives that 
$$\lim_{n\to\infty} c(\mu_n) = c(\eta) < \infty,$$
contradicting our assumption.  
\end{proof}

\begin{proof}[Proof via Theorem \ref{thm:uniform}]
  We claim that for any such $\mu$, $c(\mu) \le c_{\max}$, where $c_{\max}$ is the constant in Theorem \ref{thm:uniform}.  By \cite{em2001}, there exists some $X\in \mathcal{H}$ (in fact the following will hold for $\mu$-a.e. $X$) such that 
$$N(X,R) \sim c(\mu) R^2.$$
Hence by Theorem \ref{thm:uniform}, $c(\mu) \le c_{\max}$.  
\end{proof}

\begin{thm} \label{thm:lowerbound}
  Fix a stratum $\mathcal{H}$.  There exists a bound $\beta>0$ (depending on the stratum) such that for any ergodic $SL_2(\mathbb{R})$-invariant probability measure $\mu$ on $\mathcal{H}$,
$$c(\mu) \ge \beta.$$
\end{thm}

\begin{proof} 
  The proof strategy is the same as that of the proof of Theorem \ref{thm:upperbound} via Theorem \ref{thm:conv} above.   

Suppose, for the sake of contradiction, that $\mu_n$ is a sequence of ergodic $SL_2(\mathbb{R})$-invariant probability measures on $\mathcal{H}$, with $c(\mu_n) \to 0$.  By passing to a subsequence, and applying the equidistribution theorem \cite{emm2015} (Corollary 2.5), we can assume that $\mu_n \to \eta$, where $\eta$ is another ergodic $SL_2(\mathbb{R})$-invariant probability measure.  Then Theorem \ref{thm:conv} gives that 
$$\lim_{n\to\infty} c(\mu_n) = c(\eta),$$
which we claim is positive.  Indeed, by \cite{em2001}, the Siegel-Veech constant gives the quadratic growth rate of cylinders for an $\eta$ typical surface, and by \cite{masur1988}, the constant must be positive.   
Since we started by assuming the limit is zero, we have a contradiction.    
\end{proof}

\subsection{The set of Siegel-Veech constants}

Using Theorem \ref{thm:conv}, we can easily prove several results showing that the set of Siegel-Veech constants of all the measures for a fixed stratum is not too complicated.  
\begin{thm}
  Fix a stratum $\mathcal{H}$.  Let 
$$S(\mathcal{H}) = \{ c(\mu) : \mu \text{ an ergodic } SL_2(\mathbb{R}) \text{-invariant probability measure on } \mathcal{H}\}.$$
Then $S(\mathcal{H})$ is closed as a subset of $\mathbb{R}$.    
\end{thm}

This will follow as the $n=0$, $\mathcal{N}=\mathcal{H}$ case of Proposition \ref{prop:dim} below.    

By Eskin-Mirzakhani-Mohammadi \cite{emm2015}, the ergodic $SL_2(\mathbb{R})$-invariant probability measures on $\mathcal{H}$ are in bijection with the set of affine invariant submanifolds $\mathcal{M}$ of $\mathcal{H}$. We define $c(\mathcal{M})$ to equal $c(\mu)$, where $\mu$ is the measure corresponding to $\mathcal{M}$.  

\begin{prop}
   Fix an affine invariant submanifold $\mathcal{N}$ of $\mathcal{H}$.  Let
$$S_n(\mathcal{N}) = \{ c(\mathcal{M}) : \mathcal{M} \subset \mathcal{N}, \dim_{\mathbb{C}}\mathcal{M} \ge n \} .$$
Then $S_n(\mathcal{N})$ is closed as a subset of $\mathbb{R}$. 
\label{prop:dim}
\end{prop}

\begin{proof}
  Suppose $x\in \mathbb{R}$ and $x=\lim_{k\to\infty} c(\mathcal{M}_k)$ for some sequence $\mathcal{M}_k$ of affine invariant submanifolds, corresponding to measures $\mu_k$.  Consider the set of all affine invariant submanifolds that contain infinitely many of the $\mathcal{M}_k$, and pick an element $\mathcal{M}$ that is minimal (with respect to inclusion) in this set.  The set is non-empty (since $\mathcal{N}$ is in it), and a minimal element exists because the longest chain (with respect to inclusion) has cardinality at most $\dim_{\mathbb{C}}(\mathcal{N}) < \infty$.  Note that $\dim_{\mathbb{C}} \mathcal{M} \ge n$, hence $c(\mathcal{M}) \in S_n(\mathcal{N})$.  Now by equidistribution (\cite{emm2015}, Corollary 2.5), we can find a subsequence $j_k$ such that $\mu_{j_k}$ converges to $\mu$, where $\mu$ is the measure corresponding to $\mathcal{M}$.  

By Theorem \ref{thm:conv}, $c(\mathcal{M}) = c(\mu) = \lim_{k\to\infty} c(\mu_{j_k} ) = \lim_{k\to\infty}c(\mathcal{M}_{j_k})=x$.  Hence $x\in S_n(\mathcal{N})$, and we are done.  
\end{proof}

Given a closed $X\subset \mathbb{R}$, we define the \emph{derived set} $X^*$ to be the set obtained from $X$ by removing all the isolated points.  We let $X^{*n}=( \cdots (X^*)^* \cdots )^*$, where there are $n$ occurrences of $*$.  We define the \emph{rank}, $\operatorname{rank}(X)$, to be the smallest $n$ for which $X^{*n}=\{\}$; if no such $n$ exists, we declare the rank to be infinity.  (Note that this is a slight variation of the usual notion of \emph{Cantor-Bendixson rank}.)

\begin{thm}
  Let $\mathcal{N} \subset \mathcal{H}$ be an affine invariant submanifold with $\dim_{\mathbb{C}} \mathcal{N}=d$.  Then 
$$\operatorname{rank}S_n(\mathcal{N}) \le d-n+1.$$
  In particular, $\operatorname{rank} S(\mathcal{H})\le \dim_{\mathbb{C}} \mathcal{H}+1$.  
\end{thm}

\begin{proof}
  We argue by downwards induction on $n$.  

The base case is $n=d$.  Here the only affine invariant submanifold of $\mathcal{N}$ with dimension at least $d$ is $\mathcal{N}$ itself.  Thus $S_d(\mathcal{N})=\{c(\mathcal{N})\}$, which has rank $1$, so this gives the base case.  

For the inductive step assume the result for $n$. 

 We claim that $S_{n-1}(\mathcal{N})^* \subset S_n(\mathcal{N})$.  To see this, let $x\in S_{n-1}(\mathcal{N})^*$, which means we can find a sequence of distinct $\mathcal{M}_i$ of dimension at least $n-1$ with $\lim_i c(\mathcal{M}_i) = x$.  As in the proof of Proposition \ref{prop:dim}, using \cite{emm2015} (Corollary 2.5) we can find $\mathcal{N}$ containing all $\mathcal{M}_{j_i}$ for some subsequence $j_i$, and $c(\mathcal{N}) = \lim_i c(\mathcal{M}_{j_i}) =x$.  Since $\mathcal{N}$ contains distinct manifolds of dimension at least $n-1$, it must have dimension at least $n$ (this uses the fact that affine invariant submanifolds come from \emph{proper} immersions).  So $x=c(\mathcal{N}) \in S_n(\mathcal{N})$.  This completes the proof of the claim.  

Now for any closed sets $A\subset B \subset \mathbb{R}$, it follows immediately from our definition that $\operatorname{rank}(A) \le \operatorname{rank}(B)$.  Then, from the definition of rank, this fact, the claim above, and the inductive assumption,
$$\operatorname{rank}(S_{n-1}(\mathcal{N})) \le \operatorname{rank}(S_{n-1}(\mathcal{N})^*)+1 \le \operatorname{rank}(S_n(\mathcal{N})) +1 \le (d-n+1) + 1 = d-(n-1)+1,$$
which completes the induction.  
\end{proof}

\section{Proof of Theorem \ref{thm:conv} (Convergence of Siegel-Veech constants)}
\label{sec:convergence}

\begin{proof}[Proof of Theorem \ref{thm:conv}] 
The proof would be immediate if strata were compact, but they are not.  The recurrence-type result Proposition \ref{prop:ineq} allows us to get around this.    

The Siegel-Veech constant can be defined by 
$$c(\mu) = \frac{\int_{\mathcal{H}} \hat f d\mu }{\int_{\mathbb{R}^2}f d\lambda},$$
for any $f:\mathbb{R}^2\to \mathbb{R}$ continuous and compactly supported.  Here $\hat f:\mathcal{H} \to \mathbb{R}$ is the Siegel-Veech transform of $f$ given by 
$$\hat f(X) := \sum_{c\in \Lambda(X)} f(c),$$
where $\Lambda(X)\subset \mathbb{R}^2$ is the multi-set of \emph{holonomies} of cylinders.  The holonomy of a saddle connection $s$ is the element of $\mathbb{C}$ (which we identify with $\mathbb{R}^2$) given by integrating the 1-form $\omega$ along any of the periodic geodesics defining the cylinder.

Hence to prove Theorem \ref{thm:conv}, it suffices to show that 
$$\lim_{n\to\infty} \int_{\mathcal{H}} \hat f d\mu_n = \int_{\mathcal{H}} \hat f d\eta $$
for all such $f$.  Note that if $\hat f$ were compactly supported, this would follow immediately from the definition of weak-* convergence.  The idea is to approximate $\hat f$ by compactly supported functions, and bound the integral of the error term using integrability results from \cite{em2001}.  The key point is that we need a bound for the error term that is independent of the particular measure $\mu$.  

Let $C_K= \{X\in \mathcal{H} : \frac{1}{\ell(X)} \le K\}$.  These sets are compact.  Now let $\chi_K$ be a continuous function $\mathcal{H}\to [0,1]$ whose value is $1$ on $C_K$ and $0$ on $\mathcal{H} \backslash C_{K+1}$.  Define $\hat f_K = \hat f \cdot \chi_K$.  Note that $\hat f_K$ is compactly supported, hence
$$\lim_{n\to\infty} \int_{\mathcal{H}} \hat f_K d\mu_n = \int_{\mathcal{H}} \hat f_K d\eta.$$

It remains to show that by choosing $K$ large, we can make $\left|\int_{\mathcal{H}} \hat f d\mu - \int_{\mathcal{H}} \hat f_K d\mu\right|$ uniformly small for any choice of ergodic $SL_2(\mathbb{R})$-invariant probability measure $\mu$.  We can assume $f$, and hence $\hat f$, are non-negative, since we only need to show the equality for some $f$ for which $\int_{\mathcal{H}} \hat f d\eta$ is non-zero.  Since $f$ is compactly supported, it is dominated by some multiple of an indicator function of a large ball.  By \cite{em2001} Theorem 5.1(a), it follows that $\hat f < C/\ell^{1+\delta}$ for some $C$ and $0< \delta<1/2$ (the cited result is about counts of saddle connections, but since every cylinder is bounded by saddle connections, we get the same bound for cylinders).  Then
\begin{align}
\left|\int_{\mathcal{H}} \hat f d\mu - \int_{\mathcal{H}} \hat f_K d\mu\right| &\le \int_{\mathcal{H}\backslash C_K} \hat f d\mu \le C \int_{\mathcal{H}\backslash C_K} \frac{1}{\ell^{1+\delta}} d\mu \\
&\le C \frac{1}{K^{\delta'}} \int _{\mathcal{H}\backslash C_K} \frac{1}{\ell^{1+\delta+\delta'}} d\mu, \label{eq:error_inter}
\end{align}
where we choose $\delta'>0$ such that $\delta+\delta'<1/2$.  We introduce this extra $\delta'$ to get the $\frac{1}{K^{\delta'}}$ term in the above, which will allow us to get decay as $K\to \infty$.  By \cite{em2001} Lemma 5.5, the last integral above is finite; we need to show the somewhat stronger statement that it is bounded from above independent of the choice of $\mu$.  

The idea is to replace the integral of $\frac{1}{\ell^{1+\delta+\delta'}}$ over the whole stratum by the integral over large circles centered at a $\mu$-generic point, using Nevo's theorem, and then use Proposition \ref{prop:ineq} to bound the integrals over circles.  To apply Nevo's theorem, we need to choose a smoothing function $\phi : \mathbb{R} \to \mathbb{R}$ that is non-negative, smooth, and compactly supported (having the smoothing is fine for our purposes; Nevo's result should also be true without having to smooth).  

Now by Nevo's theorem (see \cite{em2001} Theorem 1.5) and Proposition \ref{prop:ineq}, for $\mu$ a.e. $X\in \mathcal{H}$,
\begin{align*}
\int_{\mathcal{H}} \frac{1}{\ell^{1+\delta+\delta'}} d\mu \cdot \int_{-\infty}^{\infty} \phi(t)dt &= \lim_{\tau\to\infty}\int_{-\infty}^{\infty} \phi(\tau-t) \left( \frac{1}{2\pi} \int_0^{2\pi} \frac{1}{\ell(g_t r_{\theta}X)^{1+\delta+\delta'}} d\theta  \right) dt \\
& \le b \int_{-\infty}^{\infty} \phi(t)dt. 
\end{align*}
Hence $\int_{\mathcal{H}} \frac{1}{\ell^{1+\delta+\delta'}} d\mu \le b$, for any $\mu$.  Plugging into (\ref{eq:error_inter}) gives, for all $\mu$,
\begin{align}
  \label{eq:error}
  \left|\int_{\mathcal{H}} \hat f d\mu - \int_{\mathcal{H}} \hat f_K d\mu\right|  \le C\cdot \frac{b}{K^{\delta'}}. 
\end{align}

Now we put everything together.  Fix $\epsilon>0$. Choose $K$ large so that $C~\cdot~ b/K^{\delta'} < \epsilon/3$. Now choose $N$ such that 
\begin{align}
  \label{eq:cutoff}
  \left|\int_{\mathcal{H}} \hat f_K d\mu_n - \int_{\mathcal{H}} \hat f_K d\eta\right| <\epsilon/3
\end{align}
for all $n\ge N$.  

Then by the triangle inequality, and (\ref{eq:error}), (\ref{eq:cutoff}), for $n\ge N$,
\begin{align*}
\left| \int_{\mathcal{H}} \hat f d\mu_n -  \int_{\mathcal{H}} \hat f d\eta \right| &\le \left| \int_{\mathcal{H}} \hat f d\mu_n -  \int_{\mathcal{H}} \hat f_K d\mu_n \right| + \left| \int_{\mathcal{H}} \hat f_K d\mu_n -  \int_{\mathcal{H}} \hat f_K d\eta \right| + \left| \int_{\mathcal{H}} \hat f_K d\eta -  \int_{\mathcal{H}} \hat f d\eta \right| \\
&\le C\cdot b /K^{\delta'} +\epsilon/3 + C\cdot b /K^{\delta'} < 3(\epsilon/3) = \epsilon,
\end{align*}
and we are done.  
\end{proof}

\begin{rmk}
  Some similar ideas appear in the proof of Theorem 2.4 in \cite{ems2003}. 
\end{rmk}

\section{Proof of Theorem \ref{thm:uniform} (Uniform asymptotic quadratic upper bound)}
\label{sec:uniform}

\begin{proof}[Proof of Theorem \ref{thm:uniform}]
The proof is a modification of the Eskin-Masur proof of Masur's original (non-uniform) quadratic upper bound (Theorem 5.4 in \cite{em2001}; the result was originally proved in \cite{masur1990}). 

By \cite{em2001} Proposition 3.5, there is an absolute constant $c$ such that 
$$N(X,2R)-N(X,R) \le c R^2 \int_0^{2\pi} N(g_{\log R}r_{\theta}X,4)d\theta.$$
The proof of this involves taking the indicator function of the trapezoid with vertices $(\pm 2,2), (\pm 1,1)$, and considering its Siegel-Veech transform.  Applying $g_t$ to the trapezoid makes it long and thin, and then rotating it around using $r_{\theta}$ allows one to count saddle connections whose holonomy has absolute value between $R$ and $2R$.  The right-hand side involves a radius of $4$ since the original trapezoid is contained in a ball of radius $4$.

Now we apply Theorem 5.1(a) in \cite{em2001}, and then Proposition \ref{prop:ineq} to get
  \begin{eqnarray*}
    N(X,2R)-N(X,R) &\le & c R^2 \int_0^{2\pi} N(g_{\log R}r_{\theta} X, 4) d\theta \\
&\le & c R^2 \int_0^{2\pi} \frac{c_1}{\ell(g_{\log R} r_\theta X)^{1+\delta}} d\theta \\
&\le & c c_1 R^2 \left(c_0e^{-(\log R)(1-2\delta)} \alpha(X) + b\right).
  \end{eqnarray*}
The constants $c, c_1,c_0,b$ do not depend on $X$ or $R$.   For $R\ge R_0(X)$ large, we can make $c_0e^{-(\log R )(1-2\delta)}\alpha(X) ~< ~b$, and so we get 
$$N(X,2R) - N(X,R) \le 2bcc_1 R^2.$$  A straight-forward geometric series argument then gives the desired inequality.    

We also see that the function $R_0(X)$ can be chosen to depend only on $\alpha(X)$ (in an explicit way).  The function $\alpha(X)$ is itself bounded by an explicit function of $\ell(X)$ (and the genus of the stratum).  
\end{proof}

\section{Open questions on extremal Siegel-Veech constants}
\label{sec:bound}

\begin{question}
  Given $\mathcal{H}$, what is $\sup c(\mu)$, where the sup ranges over all ergodic $SL_2(\mathbb{R})$-invariant probability measures on $\mathcal{H}$?  In particular, what are the asymptotics of this quantity as the genus of the stratum tends to infinity?  
\end{question}

We can show that the asymptotic growth rate of $\sup c(\mu)$ is somewhere between quadratic and exponential, as a function of genus (at least along some sequence of strata with genus tending to infinity).  We now explain how to get these upper and lower bounds. 

By taking branched covers of a fixed translation surface in which the preimages of all the singular points are branch points, we can exhibit a family of surfaces with genus $g\to\infty$ for which the number of cylinders of length at most $R$ is at least $kg^2R^2$ for some fixed constant $k>0$.  This shows that $\sup c(\mu)$ for these stratum is at least $kg^2$.  

On the other hand, by Theorem \ref{thm:upperbound}, the supremum is finite, and in fact any constant $c_{\max}$ for which Theorem \ref{thm:uniform} holds gives an upper bound for the supremum.   If we keep track of the $c_{\max}$ coming from the proof of Theorem \ref{thm:uniform}, we find it grows at most exponentially in the genus, but it seems hard to do better than exponential using our method.   The exponential nature of the method arises from an induction on the complexity of certain ``complexes'' of saddle connections in the proof of the generalization of Proposition \ref{prop:ineq} given in \cite{dozier2017}.

\begin{rmk}
If we work with quadratic differentials and allow simple poles (corresponding to points with cone angle $\pi$), then it is not clear that there is any bound for the analogue of $c_{\max}$ that depends only on the genus, since there are infinitely many strata in a given genus. 
\end{rmk}

The corresponding question about the smallest Siegel-Veech constant in each stratum is also interesting. 

\begin{question}
  Given $\mathcal{H}$, what is $\inf c(\mu)$, where the inf ranges over all ergodic $SL_2(\mathbb{R})$-invariant probability measures on $\mathcal{H}$?  In particular, what are the asymptotics of this quantity as the genus of the stratum tends to infinity?  
\end{question}

By Theorem \ref{thm:lowerbound}, for each stratum, $\inf c(\mu) >0$.    

\bibliography{sources}{}

\providecommand{\bysame}{\leavevmode\hbox to3em{\hrulefill}\thinspace}
\providecommand{\MR}{\relax\ifhmode\unskip\space\fi MR }
\providecommand{\MRhref}[2]{%
  \href{http://www.ams.org/mathscinet-getitem?mr=#1}{#2}
}
\providecommand{\href}[2]{#2}
\begin{thebibliography}{EMWM06}

\bibitem[Ath06]{athreya2006}
Jayadev~S. Athreya, \emph{Quantitative recurrence and large deviations for
  {T}eichmuller geodesic flow}, Geom. Dedicata \textbf{119} (2006), 121--140.
  \MR{2247652}

\bibitem[Bai07]{bainbridge2007}
Matt Bainbridge, \emph{Euler characteristics of {T}eichm\"uller curves in genus
  two}, Geom. Topol. \textbf{11} (2007), 1887--2073. \MR{2350471}

\bibitem[Bai10]{bainbridge2010}
\bysame, \emph{Billiards in {L}-shaped tables with barriers}, Geom. Funct.
  Anal. \textbf{20} (2010), no.~2, 299--356. \MR{2671280}

\bibitem[BSW16]{bsw2016}
M.~{Bainbridge}, J.~{Smillie}, and B.~{Weiss}, \emph{{Horocycle dynamics: new
  invariants and eigenform loci in the stratum H(1,1)}}, arXiv:1603.00808
  (2016).

\bibitem[{Doz}17]{dozier2017}
B.~{Dozier}, \emph{{Equidistribution of saddle connections on translation
  surfaces}}, arXiv:1705.10847 (2017).

\bibitem[EM01]{em2001}
Alex Eskin and Howard Masur, \emph{Asymptotic formulas on flat surfaces},
  Ergodic Theory Dynam. Systems \textbf{21} (2001), no.~2, 443--478.
  \MR{1827113}

\bibitem[EM13]{em2013}
Alex {Eskin} and Maryam {Mirzakhani}, \emph{{Invariant and stationary measures
  for the {${\rm SL}_2(\Bbb R)$} action on Moduli space}}, arXiv:1302.3320
  (2013).

\bibitem[EMM98]{emm1998}
Alex Eskin, Gregory Margulis, and Shahar Mozes, \emph{Upper bounds and
  asymptotics in a quantitative version of the {O}ppenheim conjecture}, Ann. of
  Math. (2) \textbf{147} (1998), no.~1, 93--141. \MR{1609447}

\bibitem[EMM15]{emm2015}
Alex Eskin, Maryam Mirzakhani, and Amir Mohammadi, \emph{Isolation,
  equidistribution, and orbit closures for the {${\rm SL}(2,\mathbb{R})$}
  action on moduli space}, Ann. of Math. (2) \textbf{182} (2015), no.~2,
  673--721. \MR{3418528}

\bibitem[EMS03]{ems2003}
Alex Eskin, Howard Masur, and Martin Schmoll, \emph{Billiards in rectangles
  with barriers}, Duke Math. J. \textbf{118} (2003), no.~3, 427--463.
  \MR{1983037}

\bibitem[EMWM06]{emm2006}
Alex Eskin, Jens Marklof, and Dave Witte~Morris, \emph{Unipotent flows on the
  space of branched covers of {V}eech surfaces}, Ergodic Theory Dynam. Systems
  \textbf{26} (2006), no.~1, 129--162. \MR{2201941}

\bibitem[EMZ03]{emz2003}
Alex Eskin, Howard Masur, and Anton Zorich, \emph{Moduli spaces of abelian
  differentials: the principal boundary, counting problems, and the
  {S}iegel-{V}eech constants}, Publ. Math. Inst. Hautes \'Etudes Sci. (2003),
  no.~97, 61--179. \MR{2010740}

\bibitem[EO01]{eo2001}
Alex Eskin and Andrei Okounkov, \emph{Asymptotics of numbers of branched
  coverings of a torus and volumes of moduli spaces of holomorphic
  differentials}, Invent. Math. \textbf{145} (2001), no.~1, 59--103.
  \MR{1839286}

\bibitem[GJ00]{gj2000}
Eugene Gutkin and Chris Judge, \emph{Affine mappings of translation surfaces:
  geometry and arithmetic}, Duke Math. J. \textbf{103} (2000), no.~2, 191--213.
  \MR{1760625}

\bibitem[Lel06]{lelievre2006}
Samuel Leli{\`e}vre, \emph{Siegel-{V}eech constants in $\mathcal{H}(2)$}, Geom.
  Topol. \textbf{10} (2006), 1157--1172. \MR{2255494}

\bibitem[Mas82]{masur1982}
Howard Masur, \emph{Interval exchange transformations and measured foliations},
  Ann. of Math. (2) \textbf{115} (1982), no.~1, 169--200. \MR{644018}

\bibitem[Mas88]{masur1988}
\bysame, \emph{Lower bounds for the number of saddle connections and closed
  trajectories of a quadratic differential}, Holomorphic functions and moduli,
  {V}ol.\ {I} ({B}erkeley, {CA}, 1986), Math. Sci. Res. Inst. Publ., vol.~10,
  Springer, New York, 1988, pp.~215--228. \MR{955824}

\bibitem[Mas90]{masur1990}
\bysame, \emph{The growth rate of trajectories of a quadratic differential},
  Ergodic Theory Dynam. Systems \textbf{10} (1990), no.~1, 151--176.
  \MR{1053805}

\bibitem[McM07]{mcmullen2007}
Curtis~T. McMullen, \emph{Dynamics of {${\rm SL}_2(\Bbb R)$} over moduli space
  in genus two}, Ann. of Math. (2) \textbf{165} (2007), no.~2, 397--456.
  \MR{2299738}

\bibitem[MMY15]{mmy2015}
Carlos Matheus, Martin M\"oller, and Jean-Christophe Yoccoz, \emph{A criterion
  for the simplicity of the {L}yapunov spectrum of square-tiled surfaces},
  Invent. Math. \textbf{202} (2015), no.~1, 333--425. \MR{3402801}

\bibitem[Sch02]{schmoll2002}
M.~Schmoll, \emph{On the asymptotic quadratic growth rate of saddle connections
  and periodic orbits on marked flat tori}, Geom. Funct. Anal. \textbf{12}
  (2002), no.~3, 622--649. \MR{1924375}

\bibitem[Vee82]{veech1982}
William~A. Veech, \emph{Gauss measures for transformations on the space of
  interval exchange maps}, Ann. of Math. (2) \textbf{115} (1982), no.~1,
  201--242. \MR{644019}

\bibitem[Vee89]{veech1989}
W.~A. Veech, \emph{Teichm\"uller curves in moduli space, {E}isenstein series
  and an application to triangular billiards}, Invent. Math. \textbf{97}
  (1989), no.~3, 553--583. \MR{1005006}

\bibitem[Vor96]{vorobets1996}
Ya.~B. Vorobets, \emph{Plane structures and billiards in rational polygons: the
  {V}eech alternative}, Uspekhi Mat. Nauk \textbf{51} (1996), no.~5(311),
  3--42. \MR{1436653}

\bibitem[Wri15]{wright2015}
Alex Wright, \emph{Translation surfaces and their orbit closures: an
  introduction for a broad audience}, EMS Surv. Math. Sci. \textbf{2} (2015),
  no.~1, 63--108. \MR{3354955}

\bibitem[Zor06]{zorich2006}
Anton Zorich, \emph{Flat surfaces}, Frontiers in number theory, physics, and
  geometry. {I}, Springer, Berlin, 2006, pp.~437--583. \MR{2261104}

\end{thebibliography}
\bibliographystyle{amsalpha}

\end{document}